\documentclass[12pt]{amsart}

\usepackage{hyperref}  

\usepackage{graphicx}
\usepackage{amsfonts}
\usepackage{amssymb}
\usepackage{amsmath}
\usepackage{amsthm}

\usepackage{xcolor}

\newtheorem{theorem}{Theorem}[section]
\newtheorem{corollary}[theorem]{Corollary}

\newtheorem{claim}{Claim}
\newtheorem{conjecture}[theorem]{Conjecture}

\newtheorem*{SGT}{Sylvester-Gallai Theorem}

\newcommand{\con}{/}

\newcommand{\cM}{\mathcal{M}}

\newcommand{\bC}{\mathbb C}
\newcommand{\cI}{\mathcal I}

\begin{document}

\sloppy

\title[A Sylvester-Gallai-type theorem]{A Sylvester-Gallai-type theorem for complex-representable matroids}

\author{Jim Geelen and Matthew E. Kroeker}\thanks{This research is partially supported by an NSERC Discovery Grant [RGPIN-2016-03886] and by an NSERC Postgraduate Scholarship [Application No. PGSD3 - 547508 - 2020]}
\address{Department of Combinatorics and Optimization,
University of Waterloo, Waterloo, Canada}

\date{\today}

\begin{abstract}
The Sylvester-Gallai Theorem states that every rank-$3$ real-representable matroid has a two-point line. We prove that, for each $k\ge 2$, every complex-representable matroid with rank at least $4^{k-1}$ has a rank-$k$ flat with exactly $k$ points. For $k=2$, this is a well-known result due to Kelly, which we use in our proof. A similar result was proved earlier by Barak, Dvir, Wigderson, and  Yehudayoff and later refined by Dvir, Saraf, and Wigderson, but we get slightly better bounds with a more elementary proof. 
\end{abstract}

\maketitle

\newtheorem*{thm}{Main Theorem}
\newtheorem*{KT}{Kelly's Theorem}
\newtheorem*{HT}{Hansen's Theorem}

\section{Introduction}
The following result was conjectured by Sylvester~\cite{Sylvester} in 1893 and proved by Gallai~\cite{Gallai} in 1944, although Melchior~\cite{Melchior} had already proved an equivalent result about line arrangements in 1941.
\begin{SGT}
Given any finite set of points in the real plane, not all on a line, there is a line in the plane that contains exactly two of them.
\end{SGT}
In matroid-theoretic terminology, which we use throughout the paper, the Sylvester-Gallai Theorem states that
every rank-$3$ real-representable matroid has a two-point line. We include a brief introduction to matroid theory, in Section~2, for the sake of geometers who might be interested in these results. 

With a view to generalizing the Sylvester-Gallai Theorem, one might hope that every rank-$4$ real-representable matroid contains a $3$-point plane.
Motzkin~\cite{motzkin} observed that this fails for the direct sum of two lines, but proved that there is always a plane that is the direct sum
of a point and a line, and he conjectured a generalization in higher dimensions. Following conventions from geometry, we say that 
a rank-$k$ flat is {\em ordinary} if it is the direct sum of a point and a rank-$(k-1)$ flat. (Note that a {\em point} is a rank-$1$ flat which, in a non-simple matroid, may have more than one element.) Motzkin's conjecture, stated below, was later proved by Hansen~\cite{Hansen}. 
\begin{HT}
Every real-representable matroid with rank at least $3$ has an ordinary hyperplane.
\end{HT}
We call a rank-$k$ flat {\em elementary} if it has exactly $k$ points. While elementary flats are clearly more natural than ordinary flats, Hansen's Theorem
itself is arguably the ``right" generalization of the Sylvester-Gallai Theorem.  Bonnice and Edelstein~\cite{BoEd} used Hansen's Theorem
to prove the following result that guarantees the existence of an elementary rank-$k$ flat in any
real-representable matroid with sufficiently high rank. Moreover, their bound of $2k-1$  is best possible, which one sees by considering the direct sum of $k-1$ three-point lines.
\begin{corollary}\label{Watson}
For each integer $k\ge 2$, every rank-$(2k-1)$ real-representable matroid has an elementary rank-$k$ flat.
\end{corollary}
The proof is quite straightforward; by Hansen's Theorem, such a matroid has a hyperplane consisting of the direct sum of a point $P$ and a rank-$(2k-3)$ flat $H$.  
Inductively we can find an elementary rank-$(k-1)$ flat in $H$ which we can extend to an elementary rank-$k$ flat by adding $P$.

The extension of the Sylvester-Gallai Theorem to complex-representability was famously resolved by Kelly~\cite{Kelly}.
\begin{KT}
Every rank-$4$ complex-representable matroid has a two-point line.
\end{KT}
The condition on the rank here is also best-possible since the rank-$3$ ternary
affine plane, AG$(2,3)$, is complex-representable but it has no two-point line. 

We propose the following analogue of Hansen's Theorem for complex-representable matroids, which is 
based on little more than optimism.
\begin{conjecture}\label{c1}
For each integer $k\ge 2$, every complex-representable matroid with rank at least $k+2$ has an ordinary rank-$k$ flat.
\end{conjecture}

Barak, Dvir, Wigderson, and Yehudayoff~\cite{Barak} proved that rank at least $\Omega(k^2)$ suffices and 
Dvir, Saraf, and Wigderson~\cite[Theorem 1.14]{Dvir} later improved the bound to $\Omega(k)$, but they did not explicitly
determine the constant. Using more elementary methods we prove the following generalization of Kelly's Theorem:
\begin{thm}\label{main_complex}
For each integer $k\ge 2$, every complex-representable matroid with rank at least $4(k-1)$ has an ordinary rank-$k$ flat.
\end{thm}

As with real-representable matroids, we get a corollary for elementary flats; we omit the easy inductive proof as it is essentially the
same as the proof of Corollary~\ref{Watson} that we sketched above. (Similar results also appear in~\cite{Barak,Dvir}.)
\begin{corollary}
For each integer $k\ge 2$, every complex-representable matroid with rank at least $4^{k-1}$ has an elementary rank-$k$ flat.
\end{corollary}

Conjecture~\ref{c1}, if true, would give the following result.
\begin{conjecture}\label{c2}
For each integer $k\ge 2$, every complex-representable matroid with rank at least $3(k-1)+1$ has an elementary rank-$k$ flat.
\end{conjecture}
Note that this conjecture, if true, is best possible, since the direct sum of $k-1$ copies of AG$(2,3)$ has rank $3(k-1)$ but does not have 
an elementary rank-$k$ flat.

The papers \cite{Barak,Dvir} prove a ``robust" version of the main theorem which seems beyond the reach of our methods. On the other hand, one strength of our technique is that it is purely matroidal, and thus extends to any minor-closed class for which the Sylvester-Gallai Theorem holds in sufficiently high rank; see Theorem~\ref{general}. For example, every orientable matroid with rank at least three has a two-point line (see, for instance,~\cite[Proposition 6.1.1]{Bjorner}), so we immediately get the following consequence.

\begin{corollary}
For each $k \geq 2$, every rank-$3(k-1)$ orientable matroid has an ordinary rank-$k$ flat.
\end{corollary}

It is conjectured that Hansen's Theorem itself holds for orientable matroids; see~\cite[Section 6.5]{Bjorner}.
For a more detailed survey on the Sylvester-Gallai Theorem and its generalizations, see Borwein and Moser~\cite{BoMo}.

\section{Matroid theory for geometers}

We use the notation and terminology of Oxley~\cite{Oxley}, although we  use ``direct sum" in a slightly unconventional way.

The change from a classical geometric perspective to a matroid-theoretic perspective is essentially just a change of terminology and is based upon the well-known translation from points in geometry to vectors in linear algebra. Given a finite set of points in $\bC^d$ we can lift each point to vector in $\bC^{d+1}$ by appending a single entry equal to one. Then any subset of the points is affinely independent if and only if their associated vectors are linearly independent. 

Let $E$ be a finite set and consider a representation $E\rightarrow \bC^{d+1}$ of the elements in the set as complex-valued vectors. We call a subset $X\subseteq E$ {\em independent} if its associated set of vectors is linearly independent.
By convention, the empty set is considered independent. Let $\cI$ denote the collection of all independent subsets of $E$. We refer to $M=(E,\cI)$ as a {\em complex-representable matroid}.  Thus, the matroid $M$ captures the combinatorial essence of the representation $E\rightarrow \bC^{d+1}$ without recording the specific coordinatization. In this paper we do not specifically need to refer to the representation itself; all of the arguments are purely matroid-theoretic, using Kelly's Theorem as the base case in an induction. We do not know of a matroid-theoretic proof of Kelly's Theorem, itself, but there are other proofs  (see \cite{Elkies},\cite{Dvir}) that are more elementary than Kelly's. 

The definition of complex-representable matroids allows for elements to be represented by the zero-vector; such elements are called {\em loops}. The definition also allows for two elements to be represented by vectors that are scalar multiples of each other; such pairs are said to be in {\em parallel}. 
We call a matroid {\em simple} if it has no loops and no parallel pairs. The matroids coming from geometry are simple, so it suffices for us to consider simple matroids. Nevertheless, parallel pairs do arise in our proof as a result of ``contraction" (or projection), but we avoid creating loops by only contracting flats.

The {\em rank} of a set $X\subseteq E$, denoted $r(X)$, is the maximum size of an independent subset of $X$.
The {\em rank of $M$}, denoted $r(M)$, is the rank of $E$.  A {\em rank-$k$ flat} of $M$ is an inclusion-wise maximal subset of $E$ with rank $k$. Thus a {\em point} is a rank-$1$ flat, a {\em line} is a rank-$2$ flat, a {\em plane} is a rank-$3$ flat, and a {\em hyperplane} is a rank-$(r(M)-1)$ flat. One convenient consequence of the matroid terminology is that the points of a rank-$3$ matroid cannot all lie on one line.

We say that a flat $F$ of a matroid is the {\em direct sum} of flats $F_1$ and $F_2$ if $F=F_1\cup F_2$ and $r(F)=r(F_1)+r(F_2)$.
This differs from the standard usage in matroid theory where ``direct sum'' is considered as a composition operation.

Finally, we require two operations, ``restriction" and ``contraction", on matroids. For a set $X\subseteq E$, the restriction of $M$ to the set $X$ is defined as you would expect and is denoted by $M|X$.
For a set $C\subseteq E$, we define the {\em contraction} of $C$ in $M$ by $M\con C:= (E\setminus C,\cI')$ where $I\subseteq E\setminus C$ is in $\cI'$ if and only if $r(C\cup I) = r(C)+|I|$. We note that $M\con C$ is a complex-representable matroid; one obtains the representation by ``projecting out" or ``quotienting on" the subspace spanned by $C$. After projection the elements of $C$ become loops so we delete them. Any elements spanned by $C$ will be loops in $M\con C$, but if $M$ is loopless and $C$ is a flat, then $M\con C$ will be loopless. Note that, if $C$ is a flat of $M$, then a set $F\subseteq E\setminus C$ is a rank-$k$ flat of $M\con C$ if and only if $F\cup C$ is a rank-$(r(C)+k)$ flat of $M$.

\section{The proof}

In this section we prove our main theorem, which we now state in full generality.

\begin{theorem}\label{general}
Let $\cM$ be a minor-closed class of matroids for which there exists an integer $C_{\cM} > 0$ such that, if $M$ is a matroid in $\cM$ with rank at least $C_{\cM}$, then $M$ has a two-point line. Then for each $k \geq 2$, every matroid in $\cM$ with rank at least $C_{\cM}(k-1)$ has an ordinary rank-$k$ flat. 
\end{theorem}

\begin{proof} Suppose that the result fails and consider a counterexample $(M,k)$ with $k$ minimum. We may assume that $M$ is simple. By hypothesis, $k\ge 3$. Let $F$ be a rank-$C_{\cM}(k-2)$ flat in $M$. Again by hypothesis, $M /F$ has a two-point line, and so there exist two rank-$(C_{\cM}(k-2)+1)$ flats $F_{1}$ and $F_{2}$ of $M$ such that $F = F_{1} \cap F_{2}$, and $F_{1} \cup F_{2}$ is a rank-$(C_{\cM}(k-2)+2)$ flat.

Let $N := M | (F_{1} \cup F_{2})$ and fix elements $x \in F_{1} \setminus F$ and $y \in F_{2} \setminus F$; note that $\{x,y\}$ is a two-point line in $N$. 
Now $N\con \{x,y\}$ has rank $C_{\cM}(k-2)$ and, so, by our choice of counterexample, $N\con\{x,y\}$ has an ordinary rank-$(k-1)$ flat.
That is, $N\con \{x,y\}$ has a rank-$(k-1)$ flat consisting of the direct sum of a rank-$(k-2)$ flat $H$ and a point $P$.
In $N$, the sets $H \cup P \cup \{x,y\}$, $H \cup \{x,y\}$, and  $P \cup \{x,y\}$ are flats of rank $(k+1)$, $k$, and $3$ respectively.

\begin{claim}\label{2point}
There exist elements $z \in P$ and $w \in \{x,y\}$ such that $\{z,w\}$ is a two-point line in $N$.
\end{claim}

\begin{proof}[Proof of claim.]
Choose $z\in P$ so that, if possible, $z\not\in F$.  We first consider the case that $z\not\in F$.
Since $P\subseteq E(N) = F_1\cup F_2$, by symmetry, we may assume that $z\in F_1$.
Since $z\in F_1\setminus F$ and $y\in F_2\setminus F$, the line spanned by $\{z,y\}$ is not contained in either $F_1$ nor $F_2$ and hence has exactly two points, as required. Therefore we may assume that $z\in F$ and, consequently, that $P\subseteq F$. The line spanned by $\{x,z\}$ is contained in $\{x\}\cup P\subseteq \{x\}\cup F$ and meets the flat $F$ in one point, namely $\{z\}$. Therefore $\{x,z\}$ is a two-point line, as required.
\end{proof}

By Claim 1 and symmetry, we may assume that there exists $z \in P$ such that $\{x,z\}$ is a two-point line in $M$. 
Choose a rank-$(k-1)$ flat $F'$ in the rank-$k$ flat $H\cup\{x,y\}$ that contains $x$ but not $y$. The following claim gives the required contradiction.
\begin{claim}\label{ord}
$F' \cup \{z\}$ is an ordinary rank-$k$ flat in $N$.
\end{claim}

\begin{proof}[Proof of claim.]
Since $F'$ is a rank-$(k-1)$ flat and $\{z\}$ is a point not contained in $F'$, it suffices to show that $F' \cup \{z\}$ is a flat.
Let $H_0$ denote the rank-$k$ flat spanned by $F' \cup \{z\}$. Since $H_0\subseteq H\cup P\cup\{x,y\}$, it suffices to show that
$H_0\cap (P\cup\{x,y\} ) = \{x,z\}$ and $H_0\cap (H\cup\{x,y\}) = F'$. Now $\{x,z\}$ is a two-point line in the plane $P\cup\{x,y\}$
and $\{x,z\}\subseteq H_0$, so, if $H_0\cap (P\cup\{x,y\} ) \neq \{x,z\}$, then $H_0$ contains the plane $P\cup\{x,y\}$, and, in particular,
$y\in H_0$. Similarly, $F'$ is a rank-$(k-1)$ flat in the rank-$k$ flat $H\cup\{x,y\}$ and $F'\subseteq H_0$, so, if
$H_0\cap (H\cup\{x,y\}) \neq F'$, then $H_0$ contains $H\cup\{x,y\}$, and, in particular, $y\in H_0$. 
In either case,  $y\in H_0$. However, if $y\in H_0$, then $H_0$ contains both the plane $P\cup\{x,y\}$ and the rank-$k$ flat $H\cup\{x,y\}$, but 
$H_0$ is a rank-$k$ flat whereas $P\cup H\cup\{x,y\}$ has rank $(k+1)$; this contradiction completes the proof.
\end{proof}

Claim 2 completes the proof.
\end{proof}

\section{Acknowledgements}
We thank James Davies, Rutger Campbell and Abhibhav Garg for interesting discussions on the topic of this paper. We also thank the anonymous referees for their helpful comments.

\end{document}